\documentclass[12pt, reqno]{amsart}
\usepackage{hyperref}
\usepackage{enumerate}
\usepackage{amssymb,amsmath,graphicx,amsthm}
\usepackage{color}

\def\T{\text}

\def\1#1{\overline{#1}}

\def\2#1{\widetilde{#1}}

\def\3#1{\widehat{#1}}

\def\4#1{\mathbb{#1}}

\def\5#1{\frak{#1}}

\def\6#1{{\mathcal{#1}}}

\def\C{{\4C}}

\def\T{\text}
\newcommand{\Om}{\Omega}

\def\C{{\Bbb C}}

\def\di{\partial}
\def\dib{\bar\partial}
\def\Label#1{\label{#1}}

\def\T{\text}

\def\1#1{\overline{#1}}

\def\2#1{\widetilde{#1}}

\def\3#1{\widehat{#1}}

\def\4#1{\mathbb{#1}}

\def\5#1{\frak{#1}}

\def\6#1{{\mathcal{#1}}}

\def\C{{\4C}}

\numberwithin{equation}{section}

\def\T{\text}

\theoremstyle{plain}

\newtheorem{theorem}{Theorem}[section]

\newtheorem{corollary}[theorem]{Corollary}

\newtheorem{lemma}[theorem]{Lemma}

\newtheorem{proposition}[theorem]{Proposition}

\theoremstyle{definition}

\newtheorem{definition}[theorem]{Definition}
\newtheorem*{Acknowledgement}{Acknowledgement}

\theoremstyle{remark}

\begin{document}
	\author[D. T. Phiet and N.V. Thu]{Dau The Phiet and Ninh Van Thu\textit{$^{1,2}$}} 
\title[Lower bounds on the Bergman metric]{Lower bounds on the Bergman metric near points of infinite type}
\address{Dau The Phiet}
		\address{Faculty of Applied Science, Ho Chi Minh City University of Technology, Vietnam National University, 268 Ly Thuong Kiet, District 10, Ho Chi Minh City, Vietnam}
		\email{dauthephiet@hcmut.edu.vn}
\address{Ninh Van Thu}
\address{\textit{$^{1}$}~Department of Mathematics, Vietnam National University at Hanoi, 334 Nguyen Trai, Thanh Xuan, Hanoi, Vietnam}
 \address{\textit{$^{2}$}~Thang Long Institute of Mathematics and Applied Sciences,
Nghiem Xuan Yem, Hoang Mai, HaNoi, Vietnam}
\email{thunv@vnu.edu.vn}
 \maketitle    
\begin{abstract}
Let $\Omega$ be a pseudoconvex domain in $\mathbb C^n$ satisfying an $f$-property for some function $f$.  We show that the Bergman metric associated to $\Omega$ has the lower bound $\tilde g(\delta_\Omega(z)^{-1})$ where $\delta_\Omega(z)$ is the distance from $z$ to the boundary $\partial\Omega$ and $\tilde g$ is a specific function defined by $f$. This refines Khanh-Zampieri's work in \cite{KZ12} with reducing the smoothness assumption of the boundary.
\end{abstract}

\section{Introduction}

Let $\Omega$ be a bounded domain in $\C^n$ with the boundary $\di\Om$, $K_\Omega(z)$ be the Bergman kernel function on $\Omega$, and $\delta_\Om(z)$ denote the distance from $z$ to the boundary of $\di\Om$. The Bergman metric associated to $\Omega$ at the point $z\in\Omega$ acting the vector $X\in T^{1,0}(\C^n)$ is defined by
\[B_\Omega(z,X) := \left(\sum\limits_{j,k=1}^n\dfrac{\di^2\log K_\Omega(z)}{\di z_j\di \bar z_k}X_j\bar X_k\right)^{1/2}.\]
It is an interesting question is to consider how fast of the Bergman metric tends to infinity uniformly at the boundary points. When $\di\Om$ is $\mathcal{C}^\infty$-smooth and $\Om$ is either strongly pseudoconvex or pseudoconvex of finite type in $\mathbb C^2$, the Bergman metric $B_\Om(z, X)$ is asymptotically equivalent to $\delta_\Om^{-1/m}(z)|X^\tau|+\delta_\Om^{-1}(z)|X^\nu|$ (see \cite{Cat89, Mc92, Die70})
where $X^\tau$ and $X^\nu$ are the tangential and normal components of $X$ and $m$ is the type
of the boundary ($m=2$ if $\Om$ is strongly pseudoconvex). For a pseudoconvex domain of finite type in $\C^n$ with $\mathcal{C}^\infty$-smooth boundary, using the subelliptic estimate for the $\dib$-Neumann problem, McNeal \cite{Mc92} gave a lower bound of this metric with rate $\delta_\Om^{-\epsilon}(z)$ for some $\epsilon>0$. This result was also obtained by Herbort \cite{Her00} and Chen \cite{Che02} by using the properties of plurisubharmonic peak functions in a H\"older space. Recently, Herbort \cite{Her14} proved that if an $(t^\epsilon\T-\tilde P)$-property (see below) holds for $\Om$ then $B_\Om(z,X)$ has the lower bound $\delta^{-\epsilon}_\Om(z)|\log(\delta_\Om(z))|^{-M}|X|$ for some $M>0$.  The novelty of the proofs by Chen and Herbort is that no smoothness assumptions of the boundary are made. It should be noted that by an amalgamation of results in \cite{Cat83,Cat87, KZ10, KZ12,Kha14}, if the boundary is smooth then the finite type condition, the subelliptic estimate for the $\dib$-Neumann problem, the existence of a family of plurisubharmonic peak functions in a H\"older space, the $(t^\epsilon\T-\tilde P)$-property, and the $(t^\epsilon\T-P)$-property (see below) are equivalent.   

For a general pseudoconvex domain $\Omega$ in $\mathbb C^n$ that is not necessary of finite type but has nevertheless smooth boundary, Khanh-Zampieri \cite{KZ10, KZ12} proved that if an $(f\T-P)$-property with $\frac{f(t)}{\ln t}\to +\infty$ holds for $\Om$ then the Bergman metric has a lower bound with the rate $ \frac{f}{\log}(\delta^{-1+\eta}_\Om(z))$ for any $\eta>0$. The aim of this paper is to improve this result by reducing the assumption of smoothness of the boundary. Here is the main result of this paper.
\begin{theorem}\label{main}
Let $\Omega$ be a bounded pseudoconvex domain in $\C^n$ with $\mathcal{C}^2$-smooth boundary $\partial\Omega$ and $\zeta$ be a boundary point. Assume that $\Omega$ has an $(f\T-P)$-property at $\zeta$ with $f$ satisfying
$$\displaystyle{g(t)^{-1}:=(\int_t^{\infty}}\dfrac{da}{af(a)}<+\infty
$$ 
for some $t\geq1$. Then there exist a neighborhood $U$ of $\zeta$ and a constant $C>0$ such that
\begin{equation}\label{eqt1.1}
B_\Omega(z,X) \geq C.\tilde g(\delta^{-1}_{\Omega}(z))|X|
\end{equation}
for any $z\in V\cap \Omega$ and $X\in T^{1,0}_z\mathbb C^n$, where the function $\tilde g$ is given by
$$
\tilde g(t)= \sqrt[4]{g(t^{1/k_0})},~t\geq 1,
$$
for some $k_0\geq 1$.
\end{theorem}

 The use of plurisubharmonic peak functions enables one to weaken the smoothness assumption on the boundary.
 
In what follows, $\lesssim$ and $\gtrsim$ denote inequalities up to a positive constant. Moreover, we will use $\approx$ for the combination of $\lesssim$ and $\gtrsim$. In addition, the superscript $^*$ denotes the inverse function.

\section{The $(f\T-P)$-property and plurisubharmonic peak functions}
We start this section by the definition of the $(f\T-P)$-property.
\begin{definition}
	\Label{d1}  For a smooth, monotonic,
	increasing function $f :[1,+\infty)\to[1,+\infty)$ with $f(t){t^{-1/2}}$ decreasing, we say that $\Omega$ has the $(f\T-P)$-property (or $f$-property for short) if there exist a neighborhood $U$ of $b\Omega$ and a family of functions $\{\phi_\delta\}$ such that
	\begin{enumerate}
		\item [(i)] the functions $\phi_\delta$ are plurisubharmonic,  $\mathcal{C}^2$ on $U$, and satisfy $-1\le \phi_\delta \le0$, and
		\item[(ii)] $i\di\dib \phi_\delta\gtrsim f(\delta^{-1})^2Id$ and $|\nabla\phi_\delta|\lesssim  \delta^{-1}$ for any  $z\in U\cap \{z\in \Omega:-\delta<r(z)<0\}$, where $r$ is a $\mathcal{C}^2$-defining function of $\Om$.
	\end{enumerate}
If the boundedness condition of $\phi_\delta$ in (i) is replaced by the self-bounded gradient condition, i.e, $i\di\dib\phi_\delta\gtrsim i\di\phi_\delta \wedge \dib\phi_\delta$, then we say that $\Om$ has the  ($f\T-\tilde P$)-property.  
\end{definition} 
 
It has been proven in \cite{Cat87, Mc92} that if $\Omega\Subset \mathbb C^n$ is of finite type, then $\Omega$ satisfies the $(t^\epsilon,P)$-property. Therefore, the estimate (\ref{eqt1.1}) holds for $\tilde g(t)=t^\delta$ for some $\delta>0$. Moreover, the $(f,P)$-property holds for a large class of infinite type pseudoconvex domains in $\mathbb C^n$, such as the following example:
 
Let $\Omega\Subset \mathbb C^n$ be a domain defined by
\begin{equation}\label{eqt1.2}
\Omega=\left\{z\in \mathbb C^n\colon \mathrm{Re}(z_n)+\sum_{j=1}^{n-1} P_j(z_j)<0 \right\},
\end{equation}
where $P_j\in \mathcal{C}^\infty(\mathbb C)$ with $\Delta P_j(z_j)\gtrsim \frac{\exp(-1/|x_j|^\alpha)}{x_j^2}$ or $\frac{\exp(-1/|y_j|^\alpha)}{y_j^2}$ for $j=1,\ldots,n-1$ and $0<\alpha<1$. Then the $(f,P)$-property holds with $f(t)=\log^{1/\alpha} t$ ( see \cite{KZ10}). Combining the results of \cite[Theorem $9.2$]{Cat87}, \cite[Proposition $3$]{KZ10}, and Theorem \ref{main},  we obtain the following corollary.
\begin{corollary}\label{Cor1}
\begin{itemize}
\item[a)] Let $\Omega$ be a bounded pseudoconvex domain of finite type in $\mathbb C^n$. Then (\ref{eqt1.1})
holds for $\tilde g(t) = t^{\delta}$ for some $\delta>0$.
\item[b)] Let $\Omega$ be defined by (\ref{eqt1.2}) with $0<\alpha< 1$. Then (\ref{eqt1.1}) holds for $\tilde g(t) = \log^{\left(\frac{1}{\alpha}-1\right)/4} t$.
\end{itemize}
\end{corollary}

The proof of Theorem \ref{main} is based on the following result about the  existence of  a family of plurisubharmonic peak functions which was recently proven by Khanh \cite{Kha13}. 

\begin{theorem}\Label{pshpeak} Under the assumption and notations of Theorem~\ref{main}, for any $\zeta\in b\Om$, there exists a $\mathcal{C}^2$ plurisubharmonic function $\psi_\zeta$ on $\Om$ which is continuous on $\overline{\Om}$ and peaks at $\zeta$ (that means, $\psi_\zeta(z)< 0$ for all $z\in\overline{\Om}\setminus\{\zeta\}$ and $\psi_\zeta(\zeta)=0$). Moreover, there are  some positive constants $c_1$ and $c_2$ such that the following holds for any constant $0<\eta<1$:
	\begin{itemize}
		\item[(i)] $|\psi_\zeta(z)-\psi_\zeta(z')|\le c_1|z-z'|^\eta$ for any $z, z'\in\overline{\Om}$; and
		\item[(ii)] $g\big((-\psi_\zeta(z))^{-1/\eta}\big)\le c_2|z-\zeta|^{-1}$ for any $z\in\overline{\Om}\setminus\{\zeta\}$.
	\end{itemize}
\end{theorem}

The function $\psi_\zeta$ above is called a plurisubharmonic peak function at the boundary point $\zeta$. The following lemma follows immediately from Theorem~\ref{pshpeak}.
\begin{corollary} \label{co2.3}
Under the assumptions of Theorem~\ref{pshpeak}, for any $\zeta\in b\Om$ there are some positive constants $c_1$ and $c_1$ such that the following holds for any constant $0<\eta<1$:
$$
-c_1|z-\zeta|^\eta\leq \psi_\zeta(z)\leq -\left(\frac{1}{g^*(c_2/|z-\zeta|)}\right)^\eta,
$$
where $\psi_\zeta$ is the plurisubharmonic peak function given in Theorem~\ref{pshpeak} and  $g^*$ is the inverse function of $g$.
\end{corollary} 

We also need a version of $L^2$-estimate for the $\bar\di-$equation that is generalized by Berndtsson \cite{Ber96}, due to Donnelly - Fefferman \cite{DF83}
\begin{proposition} [See Theorem $3.1$ in \cite{Ber96}]	\label{estimate}
Let $\Omega$ be a bounded pseudoconvex domain in $\C^n$ and let $\varphi$ be plurisubharmonic in $\Omega$. Let $\psi$ be plurisubharmonic and assume that $i\partial\bar \partial \psi\geq i \partial\psi \wedge \bar \partial \psi$ in the sense of distributions (which is equivalent to the fact that $\psi=-\log(-\nu)$ for some negative plurisubharmonic). Let $0<\nu<1$. Then for any $\bar\di$-closed (0,1)-form $g$ in $\Omega$, there is a solution $u$ to the equation $\bar\di u=g $ such that
\[\int_\Omega |u|^2e^{-\varphi + \nu\psi}dV \leq \dfrac{4}{\nu(1-\nu)^2}\int_\Omega |g|^2_{\di\bar\di \psi}e^{-\varphi+\nu\psi}dV,\]
Here $|g|^2_{\di\bar\di \psi}=\sum_{j,k} \psi^{j,\bar k} \bar g_j g_k$ (where $g=\sum g_j d\bar z_j$ and $ (\psi^{k,\bar j})$
 is the inverse of $(\psi_{j,\bar k}) = (\frac{\partial^2\psi}{\partial z_j \partial\bar z_k}) )$  denotes the length
of the form $g$ w.r.t. the K\"{a}hler metric $i\partial\bar \partial \psi$.
\end{proposition}

\section{Boundary behavior of the Bergman metric}
In order to prove the Theorem \ref{main}, we will use the following localization theorem for the Bergman metric \cite{DFH84}.
\begin{theorem} \label{localization}
Let $\Omega \subset  \mathbb C^n$ be a bounded pseudoconvex domain and let $V \Subset U$ be
open neighborhoods of a point $\zeta \in \partial\Omega$. Then there exists a constant $C \geq 1$ such
that 
\begin{equation*}
C^{-1}B_{\Omega\cap U} (z,X) \leq  B_\Omega(z, X) \leq  C B_{ \Omega\cap U} (z, X)
\end{equation*}
for any $z \in \Omega\cap V$ and  $X \in \mathbb C^n$.
\end{theorem}
Moreover, combining with the result in Theorem \ref{pshpeak}, we can now rephrase Theorem \ref{main} in a more general setting. More precisely, we have the following theorem, which generalizes \cite[Theorem $2$]{Che02}.
\begin{theorem}\label{thm}
Let $\Omega$ be a bounded pseudoconvex domain in $\C^n$, $\zeta$ be a given boundary point, $F, G:(0,1)\to (0 ,1)$ are positive convex increasing functions satisfying that the function $-\log \circ F^{-1}(-t)$ is convex on $(0,1)$. Assume that on a neighborhood $U$ of $\zeta$, there is a plurisubharmonic function $\psi_\zeta$ peaking at $\zeta$ and satisfying

\begin{equation}\label{assumption1}
-F(|z-\zeta|)\le \psi_\zeta(z) \leq -G(|z-\zeta|)  \end{equation}
for any $z\in U\cap \Omega$. Then, there exists $k_0>1$ such that
\[B_\Omega (z,X) \gtrsim\left(  \left( G^{-1}(F((3\delta_\Omega(w))^{1/k_0})) \right)^{1/2}+ \frac{3}{2} \delta_\Omega(w) \right)^{-1/2} |X|  
\]
for any $z\in U\cap \Omega$ and any $X\in T_z^{1,0}\mathbb C^n$.
\end{theorem}

Furthermore, in the case where the plurisubharmonic peak function satisfies a lower bound, we obtain the following theorem, which is a generalization of  \cite[Theorem $1$]{Her00} and \cite[Theorem $1$]{Che02}.
 \begin{theorem}\label{thm2}
Let $\Omega$ be a bounded pseudoconvex domain in $\C^n$, $\zeta$ be a given boundary point, $F: (0,1)\to (0,1)$ is positive increasing convex function satisfying that the function $-\log \circ F^{-1}(-t)$ is convex on $(0,1)$. Assume that on a neighborhood $U$ of $\zeta$, there is a plurisubharmonic function $\psi_\zeta$ peaking at $\zeta$ and satisfying

\begin{equation}\label{assumption11}
-F(|z-\zeta|)\le \psi_\zeta(z)
 \end{equation}
for any $z\in U\cap \Omega$. Then
\[\inf_{0\ne X\in T_z^{1,0}\mathbb C^n }B_\Omega (z,X)/|X|\to +\infty\]
as $z\to \zeta$.
\end{theorem}
The proofs of Theorem \ref{thm} and Theorem \ref{thm2} are adapted from the argument by Chen \cite{Che02} with precise rate of lower bounds will be given below.
 
Let $z$ be a fixed point in $\Omega\cap U$, and $\pi(z)$ be the projection of $z$ to the boundary $\partial\Omega$ such that $\pi(z)\in \partial\Omega$ is the nearest boundary point to $z$.  Since $\partial\Omega$ is $\mathcal{C}^2$-smooth, by shrinking the neighborhood $U$ if necessary, we may assume that $\pi(z)$ is uniquely defined for all $z\in \Omega\cap U$. Denote by $\rho(z):=\psi_\zeta(z))$ and by $\phi(z):=-\log (F^{-1}(-\psi_\zeta(z)))$. Notice that since $-\log \circ F^{-1}(-t)$ is convex on  and increasing on $(-1,0)$, we have that $\phi$ is plurisubharmonic on $\Omega\cap U$. Moreover, by shrinking the neighborhood $U$ if necessary, we can assume that  $F(|z-w|)<1$ for any $z\in U\cap\Omega$. 

Denote by $\delta_\Omega(z)=|z-\pi(z)|$ the Euclidian distance from $z$ to $\partial\Omega$. For $k\in \mathbb N^*$, we define on $\Omega$ a function as follows
\[g_{k,w}(z) = \chi \left(\dfrac{1}{\log k} \left(-\log \phi(z)+\log (-\log \epsilon)\right)+1\right)\log |z-w|,\]
where $\chi\in \mathcal C^\infty(\mathbb R)$ is a fixed cut-off function satisfying 
\[\chi(t)= \left\{ \begin{array}{ccc}
1 & &\text{if } t\le 0 \\
0 & &\text{if }t  \ge 1.
\end{array}\right.  \]

To prove the Theorem \ref{thm} and Theorem \ref{thm2}, we need the following lemma.
\begin{lemma}\label{lemma3.3} Let $C>2(n+1)$ be a positive constant. Then there exists a constant $k_0>1$ depends on $F$ and $n$ so that for any $w\in\Omega$ with $|w-\zeta|<\epsilon^{k_0}/2$ the following holds
\begin{enumerate}[i)]
\item $g_{k_0,w}(z) = \log|z-w|$ near $w$;
\item $4Cg_{k_0,w}(z) + \phi(z)- \log(-\log |w-z|)   $ is a plurisubharmonic function on $\Omega$.
\end{enumerate} 
\end{lemma}
\begin{proof}
 We may assume that $|w-\zeta|<\epsilon^k/2$ where $k>1$ will be determined later on. Then we have
 $$
 \rho(w)\geq -F(|w-\zeta|) \geq - F(\epsilon^k/2)>-F(\epsilon^k).
 $$
From the definition of the cut-off function $\chi$ and $g_{k,w}$, we get
\[ \left\{z\in \Omega : \rho(z) >- F(\epsilon^k))\right\}\subset  \left\{z\in\Omega : g_{k,w}(z) =\log|z-w| \right\}.\]
Thus we conclude that $g_{k,w}(z)=\log |z-w|$ near $w$. 

A computation shows that
\begin{align}
\di\bar\di g_{k,w} &= \dfrac{\log|z-w|}{\phi \log k}\Big(\chi''(.)\dfrac{\di \phi \wedge\bar \di \phi}{\phi \log k}+\chi'(.)\dfrac{\di\phi \wedge\bar \di \phi}{\phi}-\chi'(.) \di\bar \di \phi \Big)\label{levi1}\\
	&-\dfrac{\chi'(.) \log|z-w| }{\phi \log k}\Big(\di \phi\wedge \dfrac{\bar \di \log |z-w|}{\log |z-w|}+\bar\di \phi \wedge\dfrac{\di \log |z-w|}{\log |z-w|}\Big)\\
	&+	\chi(.) \di\bar\di \log|z-w|\label{levi3}.	
\end{align}
It is clear that the term in \eqref{levi3} is non-negative and thus it can be neglected. For other terms, it is sufficient to consider them in the support of $\chi'$. Moreover, we have 
\[\text{supp} \chi' \subset \left\{z\in\Omega: \rho(z)< -F(\epsilon^k)  \right\}\subset \left\{z\in\Omega: |z-w_0|\geq \epsilon^k  \right\} \]
Therefore, one obtains
\[|z-w|  \geq |z-\zeta| - |w-\zeta|  \geq \dfrac{1}{2}|z-\zeta|
\] 
on $\mathrm{supp} \chi'(.)$ since $|w-\zeta|<\epsilon^k/2$, and hence 
$$
|\phi(z)|=|\log F^{-1}(-\rho(z))|\geq |\log |z-\zeta||\geq |\log 2|z-w||.
$$

Now Cauchy-Schwarz's inequality implies that
\begin{align*} 
\pm 2\text{Re}&\left( i\dfrac{\di \phi\wedge\bar\di \log |z-w|}{\log|z-w|}\right)\geq - i \di\phi \wedge\bar \di \phi\\
&-i\di\log(-\log |z-w|) \wedge\bar \di\log(-\log |z-w|)
\end{align*}
in the distribution sense. Moreover, since $F$ is convex and increasing, it follows that $ -F^{-1}(-\rho(z))$ is a negative plurisubharmonic function, and hence $i \di\bar \di \phi\geq  i\di\phi \wedge\bar \di\phi$. 

Combining above statements, there exists a constant $C'$ (depending only on $F$ and $\chi$) so that
\[i \di\bar\di g_{k,w}(z) \geq -\dfrac{C'}{\log k}\left(i \di\bar\di \phi + i\di\bar\di (-\log(-\log |z-w|))\right)  \]                                                                                          
provided $\phi>\log 2$ on $\Omega$ and
\begin{align*} 
i\di\bar\di (-\log(-\log |z-w|)))\geq i \di\log(-\log |z-w|) \wedge\bar \di\log(-\log |z-w|).
\end{align*}
 Therefore, if we take $k_0$ big enough so that $\dfrac{C'}{\log k_0}< \dfrac{1}{4C}$, then the assertion $\mathrm{(ii)}$ follows.
\end{proof}

\begin{proof}[Proof of Theorem \ref{thm2}]
We shall follow the guidelines of \cite{Che02}. First of all, we recall that
\[B_\Omega(w,X) = K^{-1/2}(w) \sup\{|Xf(w)| :f\in H^2(\Omega), f(w) =0 \text{ and }\|f\|_\Omega\leq 1\}.\]
Let $X\in T^{1,0}(\mathbb C^n)$. Since the Bergman metric is biholomorphic invariant, we may assume without loss of generality that $X= |X|\di_{z_1}$ .
   
 Recall that $\phi= -\log(F^{-1}(-\rho))$ and define 
\begin{align*}
 \delta=\delta(\epsilon):=\sup_{z\in \overline{\Omega}, \rho(z)\geq -F(\epsilon)}|z-\zeta|. 
\end{align*}
Note that $\delta\to 0$ as $\epsilon\to 0$ since $\rho$ is a plurisubharmonic peak function. Furthermore, by Theorem \ref{localization}, we may assume without loss of generality that $U\cap \Omega=\Omega$, $\mathrm{diam}(\Omega)\leq e^{-1}$, and $\phi>\log 2$. 

Let $k_0$ be given in Lemma \ref{lemma3.3} and fix a point $w\in \Omega$ with $|w-\zeta|<\epsilon^{k_0}/2$. We now define
\begin{align*}
&&\eta_\zeta(z) &= \kappa\left(-\log(-\log|z-\zeta|)  +\log (-\log \delta^{1/2}) +1\right ),& \\
&&\psi_\zeta(z) &= \dfrac{1}{2}\left(\phi(z) -\log(-\log|z-\zeta|)\right), &\\
\text{and}
&&\varphi_w(z) &= C g_w(z) - \dfrac{1}{4}\log(-\log (|z-w|)) +\frac{1}{2}\psi_\zeta(z),&
\end{align*}
where $g_w:=g_{k_0,w}$ is given in Lemma \ref{lemma3.3} and $\kappa\in \mathcal{C}^\infty(\mathbb R)$ is a cut-off function such that
\begin{align*}
\kappa(x)=
\begin{cases}
1~\text{if}~ x< 1-\log 2, \\
0 ~\text{if}~x  > 1;
\end{cases}
\end{align*}
and $C>2(n+1)$ comes from Lemma \ref{lemma3.3}. It is easy to check by a simple computation that $\varphi_w$ is plurisubharmonic, $i\di\bar\di\psi_\zeta\gtrsim  i \di\psi_\zeta \wedge\bar\di\psi_\zeta$ and $i\di\bar\di\psi_\zeta \gtrsim \frac{1}{2} i \di \log\left(-\log|z-\zeta|\right)\wedge \bar\di\log\left(-\log|z-\zeta|\right)$. Then we obtain
\begin{align*}
|\bar\di \eta_\zeta|_{\di\bar\di \psi_\zeta}  \leq \sqrt{2}\sup|\kappa'| . 
\end{align*}
Notice that $\mathrm{supp} g_w\subset \left\{z\in \Omega\colon \rho(z)\geq - F(\epsilon)\right\}\subset \left\{z\in \Omega\colon \eta_\zeta(z)=1\right\}$ and $\text{supp} \bar \partial \eta_\zeta\subset \{z\in \Omega\colon \delta\leq  |z-\zeta|\leq\delta^{1/2}\}$. This implies that $\eta_\zeta(w)=1$ and  $\text{supp} \bar\di\eta_\zeta \cap \text{supp}g_w = \emptyset$.

Now, we apply Proposition \ref{estimate} with $\nu=\dfrac{1}{2}$, $\varphi = \varphi_w$ and $\psi=\psi_\zeta$ to solve the $\bar\di$-equation 
\begin{align*}
\bar\di u_w = (z_1-w_1)\dfrac{K_\Omega(z,w)}{K_\Omega^{1/2}(w,w)}\bar\di \eta_\zeta,
\end{align*}
on $\Omega$ with the estimate
\begin{align*}
&\int\limits_\Omega |u_w|^2 e^{-C g_w + \frac{1}{4}\log(-\log |z-w|)) }dV \\
	&\lesssim  \int\limits_{\text{supp}\bar\di\eta_w}|z_1-w_1|^2\dfrac{|K^2_\Omega(z,w)|}{|K_\Omega(w,w)|}|\bar\di\eta_\zeta|^2_{\di\bar\di \psi_\zeta} e^{-C g_w + \frac{1}{4}\log(-\log |z-w|))}dV\\
	&\lesssim \int\limits_{\text{supp}\bar\di\eta_w}|z_1-w_1|^2\dfrac{|K^2_\Omega(z,w)|}{|K_\Omega(w,w)|}|\bar\di\eta_\zeta|^2_{\di\bar\di \psi_\zeta} \left(-\log|z-w|\right)^{1/4}dV\\
&\lesssim\int\limits_{\Omega\cap \{\delta< |z-\zeta|<\delta^{1/2}\}}|z-w|^2\left(-\log|z-w|\right)^{\frac{1}{4}}\dfrac{|K_\Omega(z,w)|^2}{|K_\Omega(w,w)|}dV\\
&\lesssim\int\limits_{\Omega\cap \{\delta< |z-\zeta|<\delta^{1/2}\}}\left((|z-\zeta|+|w-\zeta|\right)\dfrac{|K_\Omega(z,w)|^2}{|K_\Omega(w,w)|}dV\\
&\lesssim\int\limits_{\Omega\cap \{\delta< |z-\zeta|<\delta^{1/2}\}}\left( \delta^{1/2}+\frac{1}{2}\epsilon^{k_0}\right) \dfrac{|K_\Omega(z,w)|^2}{|K_\Omega(w,w)|}dV\\
	&\leq C_1 \left( \delta^{1/2}+\frac{1}{2}\epsilon^{k_0}\right),
\end{align*}	
where $C_1$ is a positive constant depending only on $\sup\kappa'$. 

Since $Cg_w(z) - \dfrac{1}{4}\log(-\log|z-w|) <0$ on $\Omega$, $g_w(z)= \log|z-w| $ near $w$, and 
\begin{align*}
C g_w(z)- \dfrac{1}{4}\log(-\log|z-w|)< 2(n+1)\log |z-w|
\end{align*}	
near $w$, it follows that $u_w(w) = 0, d u_w(w)=0$, and the function  
\[f_w(z) = (z_1-w_1)\dfrac{K_\Omega(z,w)}{K^{1/2}_\Omega(w,w)}\eta_w - u_w(z)\]
is holomorphic on $\Omega$ and satisfies
\[f_w(w) =u_w(w)= 0,  Xu_w(w)=0 , \text{ and } Xf_w(w) = |X|K^{1/2}_\Omega(w,w);\]
\begin{align*}
\|f_w\|_\Omega &\leq \left\|(z_1-w_1)\dfrac{K_\Omega( .,w)}{K^{1/2	}_\Omega(w)}\eta_\zeta\right\|_\Omega +\|u_w\|_\Omega\\
 &\leq C_2\left( \delta^{1/2}+\frac{1}{2}\epsilon^{k_0}\right) +\int\limits_\Omega |u_w|^2 e^{-C g_w + \frac{1}{4}\log(-\log |z-w|)) }dV\\
& \leq C_2\left( \delta^{1/2}+\frac{1}{2}\epsilon^{k_0}\right)+ C^{1/2}_1\left( \delta^{1/2}+\frac{1}{2}\epsilon^{k_0}\right)^{1/2}\\
 &\leq C_3\left( \delta^{1/2}+\frac{1}{2}\epsilon^{k_0}\right)^{1/2},
\end{align*}
where $C_2, C_3$ are positive constants. 

Define $h_w= \dfrac{f_w}{\|f_w\|}$. Then $h_w$ is also holomorphic on $\Omega$, $h_w(w) =0$, and $\|h_w\|=1$. Therefore, we conclude that
\begin{align*}
B_\Omega(w,X)\geq \dfrac{|Xh_w(w)|}{K^{1/2}_\Omega(w)}= \dfrac{|Xf_w(w)|}{K^{1/2}_\Omega(w)\|f_w\|_\Omega}\geq C_3^{-1}\left( \delta^{1/2}+\frac{1}{2}\epsilon^{k_0}\right)^{-1/2} |X|
\end{align*}
for any $w\in \Omega\cap \{w\in \Omega\colon |w-\zeta|<\epsilon^{k_0}/2\}$ and $X\in T^{1,0}_z\mathbb C^n$. So, the proof is complete.
\end{proof}
\begin{proof}[Proof of Theorem \ref{thm}] We shall repeat the argument as in the proof of Theorem \ref{thm2}. For any $w\in \Omega\cap U$, let $w':=\pi(w)$. It means that $|w-w'|=\delta_\Omega(w)$. Then we take $\epsilon:=(3\delta_\Omega(w))^{1/k_0}$. Hence, it is clear that $|w-w'|<\epsilon^{k_0}/2$ and
$$
\delta=\sup_{z\in \overline{\Omega}, \psi_{w'}(z)\geq -F(\epsilon)}|z-w'|\leq G^{-1}(F(\epsilon)) 
$$
because $-\psi_{w'}(z)\geq G(|z-w'|)$. Therefore, we obtain 
\begin{align*}
B_\Omega(w,X)&\gtrsim \left( \delta^{1/2}+\frac{1}{2}\epsilon^{k_0}\right)^{-1/2} |X|\\
&\gtrsim\left(  \left( G^{*}(F((3\delta_\Omega(w))^{1/k_0})) \right)^{1/2}+ \frac{3}{2} \delta_\Omega(w) \right)^{-1/2}|X|,
\end{align*}
for all $w\in U\cap \Omega$ and $X\in T^{1,0}_z\mathbb C^n$, which proves the theorem.
\end{proof}

We now ready to prove Theorem \ref{main}. 
\begin{proof}[Proof of Theorem \ref{main}]
Denote by $F(t):=c_1 t^\eta $ and $G(t):=\left(\frac{1}{g^*(c_2/t)}\right)^\eta$ for all $t\geq 1$, where $0<\eta<1$. Then, a computation shows that
$$
G^{-1}\left( F((3\delta(z))^{1/k_0})\right)=c_2/g\left(\frac{1}{c_1^{1/\eta}(3\delta(z))^{1/k_0}}\right).
$$
Therefore, by Corollary \ref{co2.3} and employing Theorem \ref{thm} for $F(t)=c_1 t^\eta $ and $G(t)=\left(\frac{1}{g^*(c_2/t)}\right)^\eta$, where $\eta\in (0,1)$ is given in Corollary \ref{co2.3}, we obtain
\begin{align*}
B_\Omega(z,X) &\gtrsim \left(  \left( G^{*}(F((3\delta_\Omega(w))^{1/k_0})) \right)^{1/2}+ \frac{3}{2} \delta_\Omega(w) \right)^{-1/2}|X|\\
&\gtrsim \left(  \left( c_2/g\left(\frac{1}{c_1^{1/\eta}(3\delta(z))^{1/k_0}}\right) \right)^{1/2}+ \frac{3}{2} \delta_\Omega(w) \right)^{-1/2}|X|
\end{align*}
for any $z\in V\cap \Omega$ and $X\in T^{1,0}_z\C^n$. Moreover, by the increasing property of $g$ and decreasing property of $g(t)/t$, we conclude that
\[B_\Omega(z,X) \gtrsim \tilde{g}(\delta^{-1}_\Omega(z))|X|\]
for any $z\in V\cap \Omega$ and $X\in T^{1,0}_z\C^n$, where the function $\tilde g$ is given by
\begin{align*}
\tilde g(t)= \sqrt[4]{g(t^{1/k_0})}~\text{for every}~t\geq 1.
\end{align*}
Hence, the proof is complete. 
\end{proof}
\begin{proof}[Proof of Corollary \ref{Cor1}] 

\noindent
a) Suppose that $\Omega\Subset \mathbb C^n$ is of finite type. Then, it has been proven in \cite{Cat87, Mc92} that $\Omega$ satisfies the $(t^\epsilon,P)$-property for some $0<\epsilon<1$. Then, the function $g(t)\approx t^\epsilon, t>1$. Moreover, since $t^{\epsilon/(4k_0)}=o(t)$ as $t\to +\infty$,  it follows that  
\begin{align*}
\tilde g(t)\approx\left(  \left( c_2/g\left(\frac{t^{1/k_0}}{c_1^{1/\eta}3^{1/k_0}}\right) \right)^{1/2}+ \frac{3}{2t} \right)^{-1/2}\approx  t^{\epsilon/(4k_0)}~\text{for every}~t\geq 1.
\end{align*}
Therefore, the estimate (\ref{eqt1.1}) holds for $\tilde g(t)\approx t^\delta,t>1$, where $\delta:=\frac{\epsilon}{4k_0}>0$. 
 
\noindent
b) Let $\Omega\Subset \mathbb C^n$ be a domain defined by
\begin{equation*}
\Omega=\left\{z\in \mathbb C^n\colon \mathrm{Re}(z_n)+\sum_{j=1}^{n-1} P_j(z_j)<0 \right\},
\end{equation*}
where $P_j\in \mathcal{C}^\infty(\mathbb C)$ with $\Delta P_j(z_j)\gtrsim \frac{\exp(-1/|x_j|^\alpha)}{x_j^2}$ or $\frac{\exp(-1/|y_j|^\alpha)}{y_j^2}$ for $j=1,\ldots,n-1$ and $0<\alpha<1$. Then, the $(f,P)$-property holds with $f(t)=\log^{1/\alpha} t$ ( see \cite{KZ10}). Therefore, a computation shows that
\begin{align*}
g(t)\approx \log^{\frac{1}{\alpha}-1} (t), ~\text{for every}~t>1,
\end{align*}
and 
\begin{align*}
\tilde g(t)\approx\left(  \left( c_2/g\left(\frac{t^{1/k_0}}{c_1^{1/\eta}3^{1/k_0}}\right) \right)^{1/2}+ \frac{3}{2t} \right)^{-1/2}\approx   \log^{\left(\frac{1}{\alpha}-1\right)/4} (t),~t> 1.
\end{align*}
Hence, the estimate (\ref{eqt1.1}) holds for $\tilde g(t)=\log^{\left(\frac{1}{\alpha}-1\right)/4} (t), ~t>1$. 

Altogether, the proof of Corollary \ref{Cor1} is complete.
\end{proof}

\begin{Acknowledgement}  This work was completed when the second author was visiting the Vietnam Institute for Advanced Study in Mathematics (VIASM). He would like to thank the VIASM for financial support and hospitality. The  second author was supported by NAFOSTED under grant number 101.02-2017.311. It is a pleasure to thank Tran Vu Khanh for stimulating discussions. Especially, we would like to express our gratitude to the refrees. Their valuable comments on the first version of this paper led to significant improvements.
\end{Acknowledgement}

\end{document}